\documentclass[
	a4paper,
	leqno,
	11pt]{amsart} 
\calclayout
\usepackage{amsmath,amssymb,amsthm,mathtools}
\usepackage{mathrsfs}
\usepackage{graphicx}
\usepackage{xcolor}
\usepackage{booktabs}
\usepackage{enumitem}
\setlist[enumerate,1]{
	label={\upshape (\arabic*)},
	leftmargin=2.5em,
	labelsep=0.5em}
\setlist[enumerate,2]{
	leftmargin=*,
	labelsep=*,
	label=\alph*)}
\usepackage[psdextra]{hyperref}
\hypersetup{
    colorlinks=false,
    pdfborder={0 0 0},
}

\theoremstyle{plain}
\newtheorem{theorem}{Theorem}[section]

\newtheorem{lemma}[theorem]{Lemma}
\newtheorem{proposition}[theorem]{Proposition}

\numberwithin{theorem}{section}

\theoremstyle{definition}

\theoremstyle{remark}
\newtheorem{remark}[theorem]{Remark}

\author{Masanobu Kaneko} 
\address{Faculty of Mathematics, 
Kyushu University, 
Motooka 744, Nishi-ku,
Fuku\-oka 819-0395, 
Japan}
\email{mkaneko@math.kyushu-u.ac.jp}

\author{Masato Kuwata}
\address{Faculty of Economics, 
Chuo University, 
742-1 Higashinakano, 
Hachioji-shi, Tokyo 192-0393, 
Japan}
\email{kuwata.04f@g.chuo-u.ac.jp}


\newcommand{\C}{\mathbf C}

\newcommand{\Q}{\mathbf Q}
\newcommand{\Z}{\mathbf Z}
\renewcommand{\P}{\mathbf P}

\newcommand{\F}{\mathbf F}
\newcommand{\G}{\Gamma}
\newcommand{\Gtwo}{\Gamma^{2}}
\newcommand{\A}{\mathcal A}

\renewcommand{\u}{\upsilon}
\newcommand{\E}{\mathcal E}

\def\ord{\operatorname{ord}}

\newcommand{\SL}{\operatorname{SL}}

\def\<{\langle}\def\>{\rangle}
\def\sltwo(#1,#2;#3,#4){\mathchoice%
{\left(\hskip -\arraycolsep%
\begin{array}{rr} #1 & #2 \\ #3 & #4 \end{array}%
\hskip -\arraycolsep\right)}
{\left(\begin{smallmatrix}#1 & #2 \\ #3 & #4 \end{smallmatrix}\right)}
{\left(\begin{smallmatrix}#1 & #2 \\ #3 & #4 \end{smallmatrix}\right)}
{\left(\begin{smallmatrix}#1 & #2 \\ #3 & #4 \end{smallmatrix}\right)}
}
\def\sltwoc(#1,#2;#3,#4){\mathchoice%
{\begin{pmatrix} #1 & #2 \\ #3 & #4 \end{pmatrix}}
{\left(\begin{smallmatrix}#1 & #2 \\ #3 & #4 \end{smallmatrix}\right)}
{\left(\begin{smallmatrix}#1 & #2 \\ #3 & #4 \end{smallmatrix}\right)}
{\left(\begin{smallmatrix}#1 & #2 \\ #3 & #4 \end{smallmatrix}\right)}
}

\makeatletter
\def\@fnsymbol#1{\ensuremath{\ifcase#1\or \dagger\or \ddagger\or
   \mathsection\or \mathparagraph\or \|\or \dagger\dagger
   \or \ddagger\ddagger \else\@ctrerr\fi}}\makeatother

\numberwithin{equation}{section}

\makeatother

\newcommand{\ifrac}[2]{\resizebox{!}{1.3\height}{$\frac{#1}{#2}$}}

\def\thin{{\hskip 1pt}}

\title[2-Division on the Klein--V\'elu Septic]
{The Klein--V\'elu Septic, 2-Division, \\
and Modular Function Fields of Level 14}

\date{\today}

\begin{document}

\begin{abstract}
We study the 2-division of the Klein--V\'elu elliptic normal septic. 
Its three non-zero 2-torsion points give rise to an explicit cubic equation over the function field of $X(7)$. 
We identify the three roots of this cubic with modular functions expressed in terms of septic theta functions at 
$\tau/2,\tau$, and $2\tau$, and show that its splitting field is precisely the function field of  $X(14)$. 
This gives an explicit link between the 2-division geometry of the Klein--V\'elu septic and the passage from full level 7 to full level $14$. 
Several intermediate modular function fields in the resulting $S_3$-extension are also described.
\end{abstract}

\maketitle

% ----------------------------------------------------------------
\section{Introduction}
% ----------------------------------------------------------------

For a positive integer $N$, an elliptic curve equipped with suitable
level-$N$ data may be embedded as an elliptic normal curve of degree $N$
in $\P^{N-1}$.
Such models were studied classically by Klein~\cite{Klein} and, in low degree, by Bianchi~\cite{Bianchi}.  V\'elu~\cite{Velu:1978} later treated projective embeddings of elliptic curves with level structure from an arithmetic point of view.

The case of degree $7$ is distinguished by its close connection with
Klein's quartic.  The universal elliptic curve with full level-$7$
structure admits an elliptic normal embedding in $\P^6$, with
coefficients $a_1,a_2,a_3$ satisfying
\begin{equation}\label{eq:Klein-a}
 a_1^3a_2-a_2^3a_3-a_3^3a_1=0.
\end{equation}
After a convenient choice of modular functions $\psi$ and $\u$,
this relation takes the form
\begin{equation}\label{eq:Klein-uv}
 \psi^7=\u^4(\u-1),
\end{equation}
and the field $\C(\psi,\u)$ is the modular function field
$\A_0(\G(7))$.

Explicit equations for modular curves of level $14$, as well as
generators for their function fields, can of course be obtained by
general methods; in particular, such descriptions are already available
in the literature.  The aim of the present paper is not to give another
independent construction of $X(14)$ merely for its own sake.  Rather,
we show that the passage from level $7$ to level $14$ is already encoded
in a very concrete way in the $2$-division geometry of the
Klein--V\'elu septic.  This point of view leads naturally to septic
theta functions and gives a uniform description of the relevant
intermediate modular function fields.

Indeed, the involution $[-1]$ on the degree-$7$ model acts by
\begin{equation}\label{eq:minus-one-intro}
 (x_0:x_1:x_2:x_3:x_4:x_5:x_6)
 \longmapsto
 (x_0:x_6:x_5:x_4:x_3:x_2:x_1).
\end{equation}
Consequently, its fixed locus, apart from the origin, consists of the
three non-zero points of order $2$ and is characterized by
\[
 x_1=x_6,\qquad x_2=x_5,\qquad x_3=x_4.
\]
Intersecting this linear subspace with the universal elliptic curve and
eliminating the projective coordinates leads to the cubic
\begin{equation}\label{eq:main-cubic-intro}
 P_\u(T)
 =
 T^3+T^2-\frac{\u^2-1}{\u}T
 +\frac{(\u-1)^2}{\u},
\end{equation}
where $\u=\u(\tau)$ is the modular function defined by
\(
\u(\tau)%=\phi_1(\tau)\phi_2(\tau)^3
 =a_2(\tau)^2a_3(\tau)/a_1(\tau)^3
\).
Define two modular functions $\phi_{1}$, $\phi_{2}$ by
\begin{equation}\label{eq:def-phi12}
 \phi_1(\tau)=-a_3(\tau)/a_2(\tau),
 \quad
 \phi_2(\tau)=-a_2(\tau)/a_1(\tau).
\end{equation}
Let $g_1,g_2,g_3$ denote the three roots of $P_\u(T)$.  Our first main result identifies these roots explicitly in terms of septic theta functions.

\begin{theorem}[Theta description of the $2$-division points]
\label{thm:intro-theta}
Define
\[\left\{ \ 
\begin{aligned}
 g_1(\tau)&=\phi_1(\tau)^2/\phi_1(2\tau),\\
 g_2(\tau)&=\phi_1(\tau)\phi_1(\tau/2)\phi_1(2\tau)
            -\phi_1(2\tau)/\phi_2(\tau),\\
 g_3(\tau)&=-\phi_1(\tau)\phi_1(\tau/2)\phi_1(2\tau)
            -\phi_1(\tau)^2\phi_2(\tau)/\phi_2(2\tau).
\end{aligned}
\right.
\]
Then the cubic equation \eqref{eq:main-cubic-intro} factors as follows:
\[
 P_{\u(\tau)}(T)
 =
 (T-g_1(\tau))(T-g_2(\tau))(T-g_3(\tau)).
\]
Equivalently, we have the relations
\begin{align}
 g_1+g_2+g_3&=-1,\label{eq:sym1}\\
 g_1g_2+g_2g_3+g_3g_1&=-\u + 1/\u,\label{eq:sym2}\\
 g_1g_2g_3&=-\u- 1/\u+2.\label{eq:sym3}
\end{align}
\end{theorem}

The formulas in Theorem~\ref{thm:intro-theta} are the analytic
counterpart of the geometric $2$-division construction.  In particular,
the occurrence of the three arguments $\tau/2$, $\tau$, and $2\tau$
is not incidental: the three functions describe the three non-zero
$2$-torsion points over the full level-$7$ moduli problem.

The second main result gives the modular interpretation of the splitting
field of \eqref{eq:main-cubic-intro}.  Abstractly, adding a full
level-$2$ structure to a full level-$7$ structure gives a full
level-$14$ structure, and
\[
 \G(7)/\G(14)\simeq\mathrm{SL}_2(\F_2)\simeq S_3.
\]
The point here is that this familiar moduli-theoretic fact is realized
by the explicit cubic \eqref{eq:main-cubic-intro} and the septic-theta
functions of Theorem~\ref{thm:intro-theta}.

\begin{theorem}[The splitting field and $X(14)$]
\label{thm:intro-X14}
The modular function field $\A_0(\G(14))$ is the splitting field of
\eqref{eq:main-cubic-intro} over $\A_0(\G(7))$.  More precisely,
\[
 \A_0(\G(14))
 =
 \C(\u,\psi,g_1,\delta)
 =
 \C(\u,\psi,g_1,g_2),
\]
where $\delta$ is a suitably normalized square root of the discriminant
of $P_\u(T)$.  Consequently an affine model of $X(14)$ is given by
\begin{equation}\label{eq:X14-intro}
 \left\{ \ 
 \begin{aligned}
  &Y^7=X^4(X-1),\\
  &XZ^3+XZ^2-(X^2-1)Z+(X-1)^2=0,\\
  &W^2=X(X-1)(X^3-8X^2+5X+1).
 \end{aligned}
 \right.
\end{equation}
Here $X=\u$, $Y=\psi$, $Z=g_1$ and $W=\delta$.
\end{theorem}

Theorem~\ref{thm:intro-X14} should therefore be viewed not as a claim
that a description of the modular curve $X(14)$ itself is new, but as an explicit realization of its function field as the splitting field of the
$2$-division cubic naturally attached to the Klein--V\'elu septic.
The subgroup lattice of
\[
 \mathrm{Gal}\bigl(\A_0(\G(14))/\A_0(\G(7))\bigr)\simeq S_3
\]
then produces, from the same functions $g_i$ and $\delta$, several
intermediate modular curves.  Among them are the elliptic curve
$X_1(14)$, the genus-$2$ curve attached to
$\Gtwo\cap\G_1(7)$, the genus-$4$ curve attached to
$\G(2)\cap\G_1(7)$, and the genus-$19$ curve attached to
$\G_0(2)\cap\G(7)$.

This paper is a sequel in spirit to our earlier work~\cite{KKlevel10} on Bianchi's
elliptic quintic and modular function fields of levels $5$ and $10$.
There the main theme was the explicit description of torsion points on
the Bianchi quintic and the modular function fields obtained from them.
The degree-$7$ case has a different geometric feature which we exploit
here: the base curve is the Klein quartic, and the three non-zero
$2$-torsion points on the universal elliptic normal septic give an
explicit cubic whose full splitting field is the level-$14$ function
field.  Thus the present paper is organized around the $2$-division
problem and its $S_3$-Galois structure, rather than around a list of
function fields.

We briefly indicate how the present results fit into the existing
literature.  Elliptic normal curves with level structure and the
degree-$7$ model have a classical history going back to Klein~\cite{Klein} and
V\'elu~\cite{Velu:1978}; septic theta functions and their relation with Klein's quartic
have likewise been studied extensively (see e.g.~Elkies~\cite{Elkies:Klein-quartic}, Lachaud \cite{Lachaud}).  On the other hand, explicit
equations and generators for $X(14)$ and related modular curves are
known from general constructions (see e.g.~Yang~\cite{Yifan}).  Our contribution is to connect
these two sides explicitly: the $2$-division problem on the
Klein--V\'elu septic produces the cubic \eqref{eq:main-cubic-intro};
its roots admit the septic-theta expressions of
Theorem~\ref{thm:intro-theta}; and its splitting field and intermediate
fields recover the level-$14$ modular function fields.  To the best of
our knowledge, this explicit connection has not appeared previously.

The paper is organized as follows.
\S\ref{sec:KV} recalls only the material on septic theta
coordinates needed later and presents the Klein--V\'elu model over
$X(7)$.  \S\ref{sec:Weierstrass} gives explicit plane and
Weierstrass models of the universal elliptic curve.
\S\ref{sec:2division} studies the fixed locus of $[-1]$ and
derives the cubic \eqref{eq:main-cubic-intro}.
\S\ref{sec:theta-cubic} proves its factorization in terms of
septic theta functions.
\S\ref{sec:level14} identifies its splitting field with
$\A_0(\G(14))$ and describes selected intermediate fields.
\S\ref{sec:level49} records, separately from the main argument,
a complementary relation with modular curves of levels $7$ and $49$.
The theta identities underlying the projective model are discussed in
Appendix~\ref{sec:appendix-theta}.

% ----------------------------------------------------------------
\section{The Klein--V\'elu septic over the Klein quartic}
\label{sec:KV}
% ----------------------------------------------------------------

\subsection{Notation}

For a congruence subgroup $G\subset \mathrm{SL}_2(\Z)$, we write
$\A_0(G)$ for its field of modular functions over $\C$.
We use the standard notation $\G(N)$, $\G_1(N)$ and $\G_0(N)$.
We shall also need the inverse image in $\mathrm{SL}_2(\Z)$ of the
order-$3$ subgroup of $\mathrm{SL}_2(\F_2)$:
\[
 \Gtwo=
 \left\{
 \gamma\in\mathrm{SL}_2(\Z):
 \gamma\bmod 2\in
 \left\{
 \sltwo(1,0;0,1),\sltwo(1,1;1,0),
 \sltwo(0,1;1,1)
 \right\}
 \right\}.
\]
Thus $\G(2)\subset\Gtwo$ with index $3$.

\subsection{Septic theta coordinates}

For real parameters $p,q$, let
\[
 \theta_{(p,q)}(z,\tau)
 =
 \sum_{n\in\Z}
 e^{\pi i(n+p)^2\tau+2\pi i(n+p)(z+q)}  
 \]
be the classical theta function with characteristics.
For $k\in \frac12\Z/7\Z$, define
\begin{equation}\label{eq:theta-k}
 \theta_k(z,\tau)
 =
 i^{-1}
 \theta_{\left(\frac12-\frac{k}{7},\,\frac72\right)}(7z,7\tau).
\end{equation}
We write $\theta_k(z)$ when $\tau$ is fixed.

The following facts are standard consequences of the theory of elliptic
normal curves and are recalled here only to fix notation.

\begin{proposition}[{\cite[Th.~5.10, Th.~6.4]{KKtheta}}]\label{prop:theta-embedding}
The map
\[
 \Theta_\tau:\C/(\Z+\tau\Z)\longrightarrow\P^6,
 \qquad
 z\longmapsto
 (\theta_0(z):\theta_1(z):\cdots:\theta_6(z))
\]
is an elliptic normal embedding of degree $7$.
Its image is cut out by fourteen quadrics.
If
\[
 S=\Theta_\tau(\tau/7),\qquad T=\Theta_\tau(1/7),
\]
then $S$ and $T$ are points of order $7$, and translation by them is
given respectively by cyclic permutation and diagonal multiplication
by powers of $\zeta_7=e^{2\pi i/7}$.
Moreover, $[-1]$ is given by \eqref{eq:minus-one-intro}.
%\end{proposition}

Set
\begin{equation}\label{eq:theta-k-def}
 a_i=\theta_i(0,\tau)\qquad (i=1,2,3).
\end{equation}
Then the universal elliptic curve is given by the fourteen quadrics
\begin{equation}\label{eq:KV-14quadrics}
 \left\{
 \begin{aligned}
 a_1a_2x_i^2&=a_2^2x_{i+3}x_{i-3}-a_3^2x_{i+2}x_{i-2},\\
 a_2a_3x_i^2&=a_1^2x_{i+3}x_{i-3}-a_3^2x_{i+1}x_{i-1},
 \end{aligned}
 \right.
 \qquad i\in\Z/7\Z,
\end{equation}
together with the Klein relation \eqref{eq:Klein-a}.
The origin and a pair of $7$-torsion points are
\begin{align*}
 O&=(0:a_1:a_2:a_3:-a_3:-a_2:-a_1),\\
 S&=(-a_1:0:a_1:a_2:a_3:-a_3:-a_2),\\
 T&=(0:a_1\zeta_7^{-1}:a_2\zeta_7^{-2}:a_3\zeta_7^{-3}:
       -a_3\zeta_7^{-4}:-a_2\zeta_7^{-5}:-a_1\zeta_7^{-6}).
\end{align*}
\end{proposition}

\subsection{Different models of Klein's quartic}

Define modular functions $\psi$ and $\u$ as follows:
\begin{equation}\label{eq:uv-def}
 \psi=-a_2a_3/a_1^2,%-\frac{a_2a_3}{a_1^2},
 \qquad
 \u=a_2^2a_3/a_1^3.%\frac{a_2^2a_3}{a_1^3}.
\end{equation}
The Klein quartic \eqref{eq:Klein-a} is then equivalent to
\eqref{eq:Klein-uv}.  In particular,
\[
 \A_0(\G(7))=\C(\psi,\u),
 \quad
 \psi^7=\u^4(\u-1).
\]
Equivalently, in terms of
$\phi_1=-a_3/a_2$ and $\phi_2=-a_2/a_1$ introduced earlier, we have
\begin{equation}\label{eq:uv-phi}
 \psi=\phi_1\phi_2^2,\qquad
 \u=\phi_1\phi_2^3.
\end{equation}

With the change of variable~\eqref{eq:uv-def}, Klein--Velu's equations  \eqref{eq:KV-14quadrics} become
\begin{equation}\label{eq:KV-uv}
 \left\{ \ 
 \begin{aligned}
 &\psi\u^{3}x_i^{2}+\u^{4} x_{i+3}x_{i-3}
     -\psi^{6}x_{i+2}x_{i-2}=0,\\
 &\psi\u^2x_i^2+\u^2x_{i+3}x_{i-3}
     -\psi^4x_{i+1}x_{i-1}=0,
 \end{aligned}
 \right.
 \quad i\in\Z/7\Z.
\end{equation}

\begin{lemma}\label{lem:upsilon-hauptmodul}
The function $\u$ is a Hauptmodul for $\G_1(7)$; more precisely,
\[
 \A_0(\G_1(7))=\C(\u).
\]
With the present normalization, $\u$ is the parameter occurring
in the Kubert--Tate normal form for an elliptic curve with a
distinguished point of order $7$.
\end{lemma}

\begin{proof}
From the theta-product definition one easily sees that the Laurent series of $\u(\tau)$ has
integral powers of $q=e^{2\pi i\tau}$:
\[
 \u(\tau)
 =q^{-1}+3+4q+3q^2-5q^4-7q^5-2q^6
  +8q^7+16q^8+\cdots.
\]
In particular, $\u(\tau+1)=\u(\tau)$.  Since $\u$ is
already a modular function for $\G(7)$, and $\G_1(7)$ is generated by
$\G(7)$ and
\(
 T=\sltwo(1,1;0,1),
\)
it follows that $\u\in\A_0(\G_1(7))$.

On the other hand, since we have
\(
 [\G_1(7):\G(7)]=7,
\)
the relation
\(
 \psi^7=\u^4(\u-1)
\)
gives
\(
 [\C(\psi,\u):\C(\u)]=7
\).
Hence
\(
 \A_0(\G_1(7))=\C(\u)
\).

For comparison with the standard torsion parameter, Kubert's
parametrization~\cite{Kubert-Lang} for a point of order $7$ is the Tate normal form
\cite{Tate:1974}
\[
 y^2+(1+t-t^2)xy+(t^2-t^3)y
 =x^3+(t^2-t^3)x^2,
\]
with distinguished point $(0,0)$.  Thus, in the normalization used
here, the parameter $t$ is precisely $\u$; equivalently, the
universal curve over $X_1(7)$ is
\begin{equation}\label{eq:X17-univ-early}
\E(X_{1}(7)): \  y^2-(\u^2-\u-1)xy-\u^2(\u-1)y
 =x^3-\u^2(\u-1)x^2. \qedhere
\end{equation}
\end{proof}

% ----------------------------------------------------------------
\section{Explicit models of the universal elliptic curve}
\label{sec:Weierstrass}
% ----------------------------------------------------------------

The elliptic normal model is well adapted to the level~$7$ structure, whereas
other models are useful for arithmetic computations.

\begin{proposition}[Plane septic model]\label{prop:plane-septic}
The Klein--V\'elu curve \eqref{eq:KV-uv} is birational over
$\Q(\psi,\u)$ to the plane septic
\begin{multline}\label{eq:plane-septic}
C_7:\ 
\psi^3\bigl(\u(\u-1)^2x_0^7+x_1^7+
 \u(\u-1)^2x_2^7\bigr)
-\u^2(\u^2-3\u-3)(\u-1)^2
 x_0^3x_1x_2^3 \\
+\psi(\u-1)(3\u^3-4\u^2-3\u-1)
 x_0^2x_1^3x_2^2
-\psi^2(3\u^2-2\u-2)x_0x_1^5x_2=0.
\end{multline}
It has fourteen ordinary double points, and its normalization is of genus $1$.
The points corresponding to $O,S,T$ are
\[
 O=(0:\psi^3:-\psi\u),\quad
 S=(\psi^3:-\psi\u:-\u^2),\quad
 T=(0:\psi^3:-\psi\u\zeta_7).
\]
\end{proposition}

\begin{proof}
The equation \eqref{eq:plane-septic} may be obtained easily by eliminating $x_{3},\dots,x_{6}$ using some of the equations of \eqref{eq:KV-uv}.  
Using the Jacobian criterion, we find fourteen double points as follows:
\[
(x_{0}:x_{1}:x_{2})=\bigl(\zeta_{7}^{k}:
\zeta_{7}^{4k}\psi^{3}(1\pm\sqrt{4\u+1})/(2\u^{2}):1\bigr),
\quad k=0,1,\dots,6.
\]
The genus is then computed as 
\[  g=\frac{(7-1)(7-2)}2-14=1. \]
The displayed coordinates of $O,S,T$ are obtained by substitution.
\end{proof}

\begin{proposition}[Weierstrass models]\label{prop:Weierstrass}
The curve \eqref{eq:KV-uv} is birational over $\Q(\psi,\u)$ to
a short Weierstrass model
\[
 W_1: \ Y^2=X^3+A(\u)X+B(\u),
\]
where
\begin{align*}
A={}&-27(\u^2-\u+1)
(\u^6+229\u^5+270\u^4-1695\u^3
 +1430\u^2-235\u+1),\\
B={}&54(\u^{12}-522\u^{11}-8955\u^{10}
+37950\u^9-70998\u^8+131562\u^7\\
&\qquad-253239\u^6+316290\u^5
-218058\u^4+80090\u^3-14631\u^2+510\u+1),
\end{align*}
and also to the model
\begin{multline}\label{eq:W2}
W_2: \ 
 y^2-(\u^2-\u-1)xy
 -\u^2(\u-1)y\\
=x^3-\u^2(\u-1)x^2-5\u(\u-1)(\u^2-\u+1)(\u^3+2\u^2-5\u+1)x
\\
-\u(\u-1)\bigl(\u^9+9\u^8-37\u^7
   +70\u^6 -132\u^5
   \\
   +211\u^4-182\u^3 +76\u^2-18\u+1\bigr).
\end{multline}
Its discriminant is
\begin{equation}\label{eq:W2disc}
 \Delta(W_2)
 =
 \u(\u-1)
 \bigl(\u^3-8\u^2+5\u+1\bigr)^7.
\end{equation}
\end{proposition}

\begin{proof}
Using the point $O$ in $C_{7}$, we apply van Hoeij's algorithm~\cite{vanHoeij} to \eqref{eq:plane-septic} to obtain the first Weierstrass form $W_{1}$.  The second form $W_{2}$ is nothing but the Weierstrass form of the quotient curve $\E(X_{1}(7))/\<(0,0)\>$.  
%We will later explain that this is also a family of elliptic curve with $49$-isogeny.  
We may use standard formulas to convert $W_{2}$ to $W_{1}$ to check that they are the same curve.  Also the discriminant of $W_{2}$ may be calculated routinely.
\end{proof}

\begin{remark}
The coefficients of these Weierstrass models depend only on
$\u$, although the full level-$7$ structure is defined over
$\Q(\psi,\u)$.
The birational transformations are lengthy and are not used below.
%We therefore place them in an ancillary computer algebra file.
\end{remark}

\begin{remark}\label{rem:E4E6}
Suppose that the base field is $\C$ and that
$\u=\u(\tau)$.  The coefficients $A$ and $B$ of $W_1$
can then be written in the form
\begin{equation}\label{eq:AB-Eisenstein}
 A(\u(\tau))=-27\frac{E_4(\tau)}{Q(\tau)^4},
 \quad
 B(\u(\tau))=54\frac{E_6(\tau)}{Q(\tau)^6},
\end{equation}
where $E_4$ and $E_6$ are the normalized Eisenstein series of level 1
\begin{align*}
 E_4(\tau)&=1+240\sum_{n\ge1}\sigma_3(n)q^n
            =1+240q+2160q^2+\cdots,\\
 E_6(\tau)&=1-504\sum_{n\ge1}\sigma_5(n)q^n
            =1-504q-16632q^2+\cdots,
\end{align*}
with $q=e^{2\pi i\tau}$, and
\begin{equation}\label{eq:Q-def}
 Q(\tau)=\frac{R(\tau)}{\u(\tau)^2-\u(\tau)+1}.
\end{equation}
Here
\begin{equation}
 R(\tau)
 =\sum_{m,n\in\Z}q^{m^2+mn+2n^2} %\notag\\
 =1+2\sum_{n\ge1}\biggl(\sum_{d\mid n}
       \Bigl(\frac{-7}{d}\Bigr)\biggr)q^n
   =1+2q+4q^2+6q^4+2q^7+\cdots                         \label{eq:R-def}
\end{equation}
is the familiar weight-one theta series of discriminant $-7$; it is a
holomorphic modular form of weight $1$ on $\G_1(7)$ (equivalently, an
Eisenstein series with the quadratic character modulo $7$).  Thus
$Q$ is a meromorphic modular form of weight $1$ on $\G_1(7)$.

For completeness, \eqref{eq:AB-Eisenstein} may be checked directly from
the polynomial expressions in Proposition~\ref{prop:Weierstrass}.
Indeed, after substituting the $q$-expansion of $\u$, the two sides
have the same Fourier expansions; alternatively, multiplying by $Q^4$
and $Q^6$, respectively, reduces the assertion to identities between
modular forms of weights $4$ and $6$.  The point of
\eqref{eq:AB-Eisenstein} is that the apparently rather large polynomials
$A(\u)$ and $B(\u)$ are simply the pull-backs of the two
classical Eisenstein series, after the natural weight-one normalization
by $Q$.

If we now put
\[
 X'=\Bigl(\frac{Q(\tau)}6\Bigr)^2X,
 \quad
 Y'=\Bigl(\frac{Q(\tau)}6\Bigr)^3Y,
\]
then $W_1$ becomes
\begin{equation}\label{eq:standard-Weierstrass}
 {Y'}^2={X'}^3-\frac{E_4(\tau)}{48}X'
                 +\frac{E_6(\tau)}{864}.
\end{equation}
In particular its discriminant is
\[
 -16\biggl(
 4\Bigl(-\frac{E_4}{48}\Bigr)^3
 +27\Bigl(\frac{E_6}{864}\Bigr)^2
 \biggr)
 =\frac{E_4^3-E_6^2}{1728}
 =\eta(\tau)^{24}.
\]
Thus the model $W_1$ is, after the explicit weight-one rescaling above,
the standard Weierstrass model of the elliptic curve
$\C/(\Z+\tau\Z)$.  This observation will not be needed in the
$2$-division argument, but it explains the modular meaning of the
coefficients of $W_1$.
\end{remark}

% ----------------------------------------------------------------
\section{The $2$-division problem}
\label{sec:2division}
% ----------------------------------------------------------------

\subsection{The fixed locus of inversion}

By Proposition~\ref{prop:theta-embedding}, a non-zero point of order
$2$ satisfies
\begin{equation}\label{eq:2-div}
x_1=x_6,\qquad x_2=x_5,\qquad x_3=x_4.
\end{equation}
Let $\mathcal S(7)$ be the elliptic modular surface defined by
\eqref{eq:KV-uv}, and let $D$ be this linear fixed locus.
The normalization of $\mathcal S(7)\cap D$ parametrizes elliptic curves
with full level-$7$ structure together with a distinguished non-zero
$2$-torsion point; hence it is the modular curve associated with
$\G_0(2)\cap\G(7)$.

Substituting \eqref{eq:2-div} into \eqref{eq:KV-uv}, and 
eliminating two of the four projective coordinates, one obtains
the following more useful model.

\begin{proposition}\label{prop:2div-cubic}
The modular curve attached to $\G_0(2)\cap\G(7)$ admits an affine model
\begin{equation}\label{eq:2div-pre-cubic}
 \left\{ \ 
 \begin{aligned}
 &2\u(\u-1)t^3-(\u^2-\u-1)t^2
       -2\u t+\u=0,\\
 &\psi^7=\u^4(\u-1).
 \end{aligned}
 \right.
\end{equation}
After the change of variable $T=1/t-1$, the first equation becomes
\begin{equation}\label{eq:2div-cubic}
 P_\u(T)
 =
 T^3+T^2-\frac{\u^2-1}{\u}T
 +\frac{(\u-1)^2}{\u}=0.
\end{equation}
\end{proposition}

\begin{proof}
Letting $x_{1}=x_{6}$, $x_{2}=x_{5}$, and $x_{3}=x_{4}$ in \eqref{eq:KV-uv}, and using two of them, we obtain
\begin{equation}\label{eq:x2x3}
x_{2}=-\frac{(\psi^2 x_{0} - \u^2 x_{1}) x_{1}}
{\psi (\psi^2 x_{0} - \u x_{1})},
\quad
x_{3} = -\frac{\psi^2 (\psi^2 x_{0}-\u^2 x_{1}) x_{1}^2}
{\u (\psi^2 x_{0}-\u x_{1})^2}.
\end{equation}
Substituting these back into the other equations, we obtain one equation:
\[
2 \psi^5 x_{1}^3 - \u^2 (\u^2 - \u - 1) x_{0} x_{1}^2 - 2 \psi^2 \u^2  x_{1} x_{0}^2 + \psi^4 \u x_{0}^3 =0.
\]
Furthermore, substituting $t=(\u/\psi^2)(x_{1}/x_{0})$, we obtain the first equation of \eqref{eq:2div-pre-cubic}.  Finally, the substitution $T=1/t-1$
transforms it into the form \eqref{eq:2div-cubic}.
\end{proof}

The cubic \eqref{eq:2div-cubic} already exhibits $X_1(14)$ as a
triple cover of $X_1(7)\simeq\P^1_\u$.

\begin{theorem}\label{thm:2torsion-coordinates}
Let $g_1,g_2,g_3$ be the three roots of \eqref{eq:2div-cubic}.
Then the three non-trivial points of order $2$ on the Klein--V\'elu
septic are
\begin{equation}\label{eq:2torsion-coordinates}
\left(
g_i+1:\ifrac{\psi^2}{\u}:
\ifrac{\u-(g_i+1)}{g_i}\ifrac{\psi}{\u}:
\ifrac{\u-(g_i+1)}{g_i^2}\ifrac{\psi^4}{\u^3}:
\ifrac{\u-(g_i+1)}{g_i^2}\ifrac{\psi^4}{\u^3}:
\ifrac{\u-(g_i+1)}{g_i}\ifrac{\psi}{\u}:
\ifrac{\psi^2}{\u}
\right),
\end{equation}
for $i=1,2,3$.
\end{theorem}

\begin{proof}
Looking at the proof of Proposition~\ref{prop:2div-cubic}, we have $T+1=(\psi^2/\u)(x_{0}/x_{1})$.  Thus, if $T=g_{i}$, then $x_{0}:x_{1}=g_{i}+1:\psi^{2}/\u$.  Then $x_{2}$ and $x_{3}$ are obtained using \eqref{eq:x2x3}.
\end{proof}

\begin{remark}\label{rmk:Hulek-Craig}
We obtain another model of the modular curve attached to $\G_0(2)\cap\G(7)$ by eliminating $\u$ and $\psi$ from \eqref{eq:2-div}:
\begin{multline}\label{eq:HC-like}
x_{0}^{11} x_{1} x_{2}^{5}-5 x_{0}^{8} x_{1}^{5} x_{2}^{4}
+5 x_{0}^{7} x_{1}^{4} x_{2}^{6}-x_{0}^{6} x_{1}^{10} x_{2}
-2 x_{0}^{6} x_{1}^{3} x_{2}^{8}+10 x_{0}^{5} x_{1}^{9} x_{2}^{3}
\\
+x_{0}^{5} x_{1}^{2} x_{2}^{10}-17 x_{0}^{4} x_{1}^{8} x_{2}^{5}
-x_{0}^{4} x_{1} x_{2}^{12}+x_{0}^{3} x_{1}^{14}
+11 x_{0}^{3} x_{1}^{7} x_{2}^{7}-x_{0}^{3} x_{2}^{14}
\\
-6 x_{0}^{2} x_{1}^{13} x_{2}^{2}-4 x_{0}^{2} x_{1}^{6} x_{2}^{9}+12 x_{0} x_{1}^{12} x_{2}^{4}+4 x_{0} x_{1}^{5} x_{2}^{11}-8 x_{1}^{11} x_{2}^{6}=0.
\end{multline}
The normalization of this plane curve is of genus~$19$.  In the level-$5$ case, this curve corresponds the Hulek-Craig model of Bring's curve \cite{Hulek:1993}, which may be viewed as a model of the modular curve associated with $\Gamma_{0}(2)\cap \Gamma(5)$. 
\end{remark}

\subsection{Septic theta functions attached to the roots}

Recall the definition~\eqref{eq:def-phi12}:
\begin{equation}\label{eq:phi12}
 \phi_1(\tau)=-a_3(\tau)/a_2(\tau),
 \qquad
 \phi_2(\tau)=-a_2(\tau)/a_1(\tau).
\end{equation}
Then we obtain another affine form of Klein's quartic.
\begin{equation}\label{eq:phi-klein}
 \phi_1\phi_2^3-\phi_1^3\phi_2^2=1,
\end{equation}
Moreover, we have
\[
 \psi=\phi_1\phi_2^2,\quad
 \u=\phi_1\phi_2^3.
\]  Their initial expansions are
\begin{align*}
\phi_1(\tau)&=q^{-1/7}+q^{13/7}-q^{20/7}+q^{48/7}
-q^{55/7}+q^{62/7}-q^{69/7}-q^{90/7}+2q^{97/7}+\cdots,\\
\phi_2(\tau)&=q^{-2/7}+q^{5/7}-q^{33/7}+q^{47/7}
+q^{54/7}-q^{68/7}-q^{75/7}-q^{82/7}+2q^{96/7}+\cdots.
\end{align*}
Define
\begin{align}
g_1(\tau)&=\phi_1(\tau)^2/\phi_1(2\tau),\label{eq:g1-def}\\
g_2(\tau)&=\phi_1(\tau)\phi_1(\tau/2)\phi_1(2\tau)
-\phi_1(2\tau)/\phi_2(\tau),\label{eq:g2-def}\\
g_3(\tau)&=-\phi_1(\tau)\phi_1(\tau/2)\phi_1(2\tau)
-\phi_1(\tau)^2\phi_2(\tau)/\phi_2(2\tau).\label{eq:g3-def}
\end{align}
Their leading terms are
\begin{align*}
g_1&=1+2q^2-2q^3-2q^5+4q^7+4q^9-6q^{10}-2q^{11}-4q^{12}+10q^{14}+\cdots,\\
g_2&=q^{-1/2}-1+q^{1/2}+q^{3/2}-q^2+q^3-q^{7/2}+q^5-q^{11/2}
+2q^{13/2}-2q^7+\cdots,\\
g_3&=-q^{-1/2}-1-q^{1/2}-q^{3/2}-q^2+q^3+q^{7/2}+q^5+q^{11/2}
-2q^{13/2}-2q^7+\cdots.
\end{align*}
The proof that these are precisely the roots of \eqref{eq:2div-cubic}
is given in the next section.

% ----------------------------------------------------------------
\section{A cubic relation for septic theta functions}
\label{sec:theta-cubic}
% ----------------------------------------------------------------

This section contains the analytic input of the paper.
The purpose is to prove Theorem~\ref{thm:intro-theta} independently
of the geometric elimination in the preceding section.

\subsection{The first symmetric relation}

For brevity, put
\begin{alignat*}{3}
 x_-&=\phi_1(\tau/2),&\quad x_0&=\phi_1(\tau),&\quad
 x_+&=\phi_1(2\tau),
\\
 y_-&=\phi_2(\tau/2),& y_0 &=\phi_2(\tau),&
 y_+&=\phi_2(2\tau).
\end{alignat*}
Then
\[
 g_1=\frac{x_0^2}{x_+},\qquad
 g_2=x_-x_0x_+-\frac{x_+}{y_0},\qquad
 g_3=-x_-x_0x_+-\frac{x_0^2y_0}{y_+}.
\]
Hence the terms involving $\tau/2$ cancel, and it remains to prove
\begin{equation}\label{eq:first-symmetric-reduced}
 \frac{x_0^2}{x_+}-\frac{x_+}{y_0}-\frac{x_0^2y_0}{y_+}=-1.
\end{equation}

We use the multiplicative theta notation
\[
 \theta(z;Q):=(z;Q)_\infty(Q/z;Q)_\infty,
 \qquad
 (z;Q)_\infty:=\prod_{n=0}^{\infty}(1-zQ^n),
\]
and, as usual,
\[
 \theta(z_1,\ldots,z_m;Q)
 :=\prod_{j=1}^m\theta(z_j;Q).
\]
We shall use the Weierstrass three-term theta identity
\begin{equation}\label{eq:weierstrass-three-term}
\theta(ac,a/c,bd,b/d;Q)
 -\theta(ad,a/d,bc,b/c;Q)
  =\frac{b}{c}\,
   \theta(cd,c/d,ab,a/b;Q)
\end{equation}
together with
\begin{equation}\label{eq:theta-inversion}
 \theta(Q/z;Q)=\theta(z;Q).
\end{equation}
For \eqref{eq:weierstrass-three-term}, see, for example, Koornwinder~\cite{KoornwinderTheta};
this is the multiplicative form of the classical Weierstrass addition
formula.

Put
\[
 \vartheta_r:=\theta(q^r;q^{14})\qquad(1\le r\le6).
\]
The Jacobi triple product, applied to the definition~\eqref{eq:theta-k}, gives
\begin{equation}\label{eq:ak-product}
\begin{aligned}
 a_1(\tau)
   &=q^{25/56}(q^7;q^7)_\infty\,
      \vartheta_1\vartheta_6,\\
 a_2(\tau)
   &=-q^{9/56}(q^7;q^7)_\infty\,
      \vartheta_2\vartheta_5,\\
 a_3(\tau)
   &=q^{1/56}(q^7;q^7)_\infty\,
      \vartheta_3\vartheta_4.
\end{aligned}
\end{equation}
Consequently,
\begin{equation}\label{eq:x0y0-vartheta}
 x_0=q^{-1/7}
   \frac{\vartheta_3\vartheta_4}
        {\vartheta_2\vartheta_5},
 \qquad
 y_0=q^{-2/7}
   \frac{\vartheta_2\vartheta_5}
        {\vartheta_1\vartheta_6}.
\end{equation}
Replacing $\tau$ by $2\tau$ and using \eqref{eq:theta-inversion}, we similarly obtain
\begin{equation}\label{eq:xplus-yplus-vartheta}
 x_+=q^{-2/7}\frac{\vartheta_6}{\vartheta_4},
 \qquad
 y_+=q^{-4/7}\frac{\vartheta_4}{\vartheta_2}.
\end{equation}
It follows that
\begin{equation}\label{eq:three-quotients-vartheta}
 \frac{x_0^2}{x_+}=
 \frac{\vartheta_3^2\vartheta_4^3}
      {\vartheta_2^2\vartheta_5^2\vartheta_6},\quad
 \frac{x_+}{y_0}=
 \frac{\vartheta_1\vartheta_6^2}
      {\vartheta_2\vartheta_4\vartheta_5},\quad
 \frac{x_0^2y_0}{y_+}=
 \frac{\vartheta_3^2\vartheta_4}
      {\vartheta_1\vartheta_5\vartheta_6}.
\end{equation}

We now apply \eqref{eq:weierstrass-three-term}, always with $Q=q^{14}$.
Taking $(a,b,c,d)=(q^6,q^5,q^4,1)$ gives
\begin{equation}\label{eq:weierstrass-specialization-one}
\vartheta_2\vartheta_4\vartheta_5^2
-\vartheta_1\vartheta_5\vartheta_6^2
=q\vartheta_1\vartheta_3\vartheta_4^2.
\end{equation}
Taking $(a,b,c,d)=(q^4,q^2,q,1)$ gives
\[
\vartheta_2^2\vartheta_3\vartheta_5
-\vartheta_1\vartheta_3\vartheta_4^2
=q\vartheta_1^2\vartheta_2\vartheta_6,
\]
or
\begin{equation}\label{eq:weierstrass-specialization-two}
\vartheta_1\vartheta_4^2-\vartheta_2^2\vartheta_5
=-q\frac{\vartheta_1^2\vartheta_2\vartheta_6}{\vartheta_3}.
\end{equation}
Putting the three expressions in \eqref{eq:three-quotients-vartheta}
over a common denominator and applying
\eqref{eq:weierstrass-specialization-one} and
\eqref{eq:weierstrass-specialization-two}, we obtain
\eqref{eq:first-symmetric-reduced}. Thus
\begin{equation}\label{eq:first-symmetric}
g_1+g_2+g_3=-1.
\end{equation}

\subsection{The cubic equation for $g_1$}

The aim is to prove
\[  P_\upsilon(g_1)= g_1^3+g_1^2-\frac{\u^2-1}{\u} g_1+\frac{(\u-1)^2}{\u}=0.\]

Let
\begin{gather*}
a=a_1(\tau),\quad b=a_2(\tau),\quad c=a_3(\tau),\\
\noalign{\noindent and put}
b_+=a_2(2\tau),\quad c_+=a_3(2\tau).
\end{gather*}
Then
\[
\upsilon=\frac{b^2c}{a^3},\quad
g_1=-\frac{c^2b_+}{b^2c_+}.
\]

We recall the standard consequence of the transformation formula for
theta functions with rational characteristics that $a_1,a_2,a_3$ are
holomorphic modular forms of weight $1/2$ on $\Gamma(7)$ with a common
finite-order multiplier. Consequently the functions $a_i(\tau)$ and
$a_i(2\tau)$ are simultaneously modular on
\(
H:=\Gamma_0(2)\cap\Gamma(7).
\)
Indeed, if
\(
\gamma=\sltwo(A,B;C,D)\in H,
\)
then
\[
\gamma^{(2)}
:=
\begin{pmatrix}2&0\\0&1\end{pmatrix}
\gamma
\begin{pmatrix}1/2&0\\0&1\end{pmatrix}
=
\begin{pmatrix}A&2B\\ C/2&D\end{pmatrix}
\in\Gamma(7),
\]
and $2\gamma\tau=\gamma^{(2)}(2\tau)$.
Substitution gives
\[
P_\upsilon(g_1)=\frac{F(\tau)}{c_+^3a^3b^6c},
\]
where
\begin{multline}\label{eq:F-theta-identity}
F=-b_+^3a^3c^7+b_+^2c_+a^3b^2c^5-b_+c_+^2a^6b^2c^2 
+b_+c_+^2b^6c^4+c_+^3a^6b^4-2c_+^3a^3b^6c+c_+^3b^8c^2.
\end{multline}
Every monomial in $F$ has degree $10$ in the theta constants at $\tau$
and degree $3$ in those at $2\tau$. Hence $F$ is a holomorphic modular
form of weight $13/2$ on $H$, with a finite-order multiplier.

To avoid using a half-integral-weight valence formula, take a sufficiently
large even power of $F$ so that the weight is integral and the multiplier
is trivial, and apply the ordinary valence formula to that power. Dividing
the resulting inequality by the exponent gives the sufficient bound
\[
\frac{13/2}{12}[\mathrm{SL}_2(\mathbf Z):H].
\]
Since
\[
[\mathrm{SL}_2(\mathbf Z):H]
=3[\mathrm{SL}_2(\mathbf Z):\Gamma(7)]
=1008,
\]
the bound is $546$. The width of the cusp $\infty$ for $H$ is $7$, so it
is enough to verify
\[
\operatorname{ord}_qF>\frac{546}{7}=78.
\]

From the definition of the septic theta constants,
\begin{equation}\label{eq:ak-exact-series}
a_k(\tau)=
\sum_{n\in\mathbf Z}(-1)^{n+k+1}
q^{(14n+7-2k)^2/56},
\qquad k=1,2,3.
\end{equation}
Using \eqref{eq:ak-exact-series}, an exact symbolic expansion of
\eqref{eq:F-theta-identity} gives
\[
F(\tau)=O(q^{90}).
\]
Thus $F=0$, and hence
\[
P_{\upsilon(\tau)}(g_1(\tau))=0.
\]

\subsection{The product of the remaining roots}

Write
\begin{alignat*}{3}
a_-&=a_1(\tau/2),\quad &b_-&=a_2(\tau/2),&\quad c_-&=a_3(\tau/2),
\\
a_+&=a_1(2\tau),&b_+&=a_2(2\tau),&c_+&=a_3(2\tau).
\end{alignat*}
Put
$A=x_-x_0x_+$, $B=x_+/y_0$, and
$C=x_0^2y_0/y_+$.
Then, we have $g_2=A-B$, and $g_3=-A-C$.
It is therefore enough to prove
\begin{equation}\label{eq:ABC-product}
(A-B)(A+C)=x_0^3y_0x_+.
\end{equation}
Using $x_0=-c/b$, $y_0=-b/a$, and the analogous formulas at
$\tau/2$ and $2\tau$, this is equivalent to $K(\tau)=0$, where
\begin{equation}\label{eq:K-theta-identity}
K=a^2b_-c_-c_+-aa_+b_-^2c+ac\,c_-^2c_+ -a_+b_-c^2c_-+b_-^2b_+c^2.
\end{equation}
Each monomial in $K$ has degree $2$ in the theta constants at $\tau$,
degree $2$ in those at $\tau/2$, and degree $1$ in those at $2\tau$.
The three sets of theta constants are simultaneously modular on
$\Gamma(14)$; hence $K$ is a holomorphic modular form of weight $5/2$
on $\Gamma(14)$ with a finite-order multiplier.

As above, take a sufficiently large even power and apply the ordinary
valence formula. This gives the sufficient bound
%\[
%[\mathrm{SL}_2(\mathbf Z):\Gamma(14)].
%\]
%Since
\[
\frac{5/2}{12}[\mathrm{SL}_2(\mathbf Z):\Gamma(14)]
=\frac{5}{24}\cdot14^3\left(1-\frac1{2^2}\right)
\left(1-\frac1{7^2}\right)=420.
\]
The width of the cusp $\infty$ is $14$, so it is
enough to verify
$\ord_qK>\frac{420}{14}=30$.
Using \eqref{eq:ak-exact-series}, with $\tau$ replaced by $\tau/2$ and
$2\tau$ where appropriate, an exact symbolic computation gives
$K(\tau)=O(q^{40})$.
Thus $K=0$, and therefore
\begin{equation}\label{eq:g2g3-intermediate}
g_2g_3=-x_0^3y_0x_+.
\end{equation}

From the Klein relation \textup{(4.8)} and \textup{(2.5)},
$x_0y_0^3=\upsilon$, and $x_0^3y_0^2=\upsilon-1.$
Since $g_1=x_0^2/x_+$, we have
$x_0^3y_0x_+=(\upsilon-1)^2/\upsilon g_1$.
Hence we have
\begin{equation}\label{eq:g2g3-v}
g_2g_3=-\frac{(\upsilon-1)^2}{\upsilon g_1}.
\end{equation}
Multiplying by $g_1$, we obtain
\[
g_1g_2g_3
=-\frac{(\upsilon-1)^2}{\upsilon}
=-\upsilon-\frac1{\upsilon}+2.
\]
Finally, by \eqref{eq:first-symmetric},
\[
g_1g_2+g_2g_3+g_3g_1
=g_1(g_2+g_3)+g_2g_3
=-g_1-g_1^2-\frac{(\upsilon-1)^2}{\upsilon g_1}.
\]
Dividing $P_\upsilon(g_1)=0$ by $g_1$ gives
\[
g_1^2+g_1-\frac{\upsilon^2-1}{\upsilon}
+\frac{(\upsilon-1)^2}{\upsilon g_1}=0.
\]
Therefore
\[
g_1g_2+g_2g_3+g_3g_1=-\upsilon+\frac1{\upsilon}.
\]
Together with \eqref{eq:first-symmetric}, these identities prove
\[
P_\upsilon(T)=(T-g_1)(T-g_2)(T-g_3),
\]
and complete the proof of Theorem~\ref{thm:intro-theta}.

% ----------------------------------------------------------------
\section{Modular function fields of level $14$}
\label{sec:level14}
% ----------------------------------------------------------------

The preceding sections explain why the level-$14$ function fields
should be organized by the Galois theory of the $2$-division cubic.
This is the point at which we pass from the geometry of one elliptic
curve to the tower of modular curves.

Put
\begin{equation}\label{eq:delta}
 \delta(\tau)
 =
 \frac{\u(\tau)^2}
      {2(\u(\tau)-1)}
 (g_1-g_2)(g_2-g_3)(g_3-g_1).
\end{equation}
Then
\begin{equation}\label{eq:delta-square}
 \delta(\tau)^2
 =
 \u(\tau)(\u(\tau)-1)
 \bigl(\u(\tau)^3-8\u(\tau)^2
      +5\u(\tau)+1\bigr).
\end{equation}

\subsection{The principal congruence subgroup}

\begin{theorem}\label{thm:X14}
The field $\A_0(\G(14))$ is the splitting field of
$P_\u(T)$ over $\A_0(\G(7))$.  In particular,
\[
 \A_0(\G(14))
 =
 \C(\u,\psi,g_1,\delta)
 =
 \C(\u,\psi,g_1,g_2).
\]
The equations \eqref{eq:X14-intro} define an affine model of $X(14)$,
whose smooth projective model has genus $49$.
\end{theorem}

\begin{proof}
The functions $\u$ and $\psi$ generate $\A_0(\G(7))$.
Over the moduli problem with full level-$7$ structure, adjoining one
non-zero $2$-torsion point amounts to passing from $\G(7)$ to
$\G_0(2)\cap\G(7)$; this is an extension of degree $3$.
Thus $g_1$ generates the corresponding cubic subextension.

Reduction modulo $2$ gives
\[
 \G(7)/\G(14)\simeq \mathrm{SL}_2(\F_2)\simeq S_3.
\]
Adjoining all three non-zero $2$-torsion points is therefore equivalent
to adjoining a full level-$2$ structure, and, by the Chinese remainder
theorem, to passing from full level $7$ to full level $14$.
Consequently the normal closure of
$\A_0(\G(7))(g_1)/\A_0(\G(7))$ is exactly
$\A_0(\G(14))$.

The unique quadratic subextension corresponds to the alternating
subgroup of $S_3$ and is generated by the discriminant $\delta$.
Hence
\[
 \A_0(\G(14))
 =\C(\u,\psi,g_1,\delta)
 =\C(\u,\psi,g_1,g_2).
\]
Equations \eqref{eq:X14-intro} are simply the Klein relation, the
$2$-division cubic, and the discriminant relation
\eqref{eq:delta-square}.
\end{proof}

\subsection{Intermediate fields}

Rather than treating every subgroup in a separate long subsection,
we summarize the field-theoretic pattern first.
The most useful intermediate fields between $\A_{0}(\G(14))$ and $\A_{0}(\G_{0}(7))$ are listed below.

\begin{center}
\begin{tabular}{lll}
\toprule
Group $G$ & $\A_0(G)$ & genus\\
\midrule
$\G_0(2)\cap\G(7)$
 & $\C(\psi,\u,g_1)$ & $19$\\
$\Gtwo\cap\G(7)$
 & $\C(\psi,\u,\delta)$ & $17$\\
$\G(2)\cap\G_1(7)$
 & $\C(\u,g_1,g_2)$ & $4$\\
$\Gtwo\cap\G_1(7)$
 & $\C(\u,\delta)$ & $2$\\
$\G_1(14)=\G_0(2)\cap\G_1(7)$
 & $\C(\u,g_1)$ & $1$\\
$\G(2)\cap\G_0(7)$
 & $\C(\sqrt{j_{7}},j_{14})=\C(\sqrt{j_{7}},g_{1})$ & $2$\\
$\Gtwo\cap\G_0(7)$
 & $\C(\sqrt{j_{7}})$ & $0$\\
$\G_0(14)=\G_0(2)\cap\G_0(7)$
 & $\C(j_{7},j_{14})=\C(j_{7},g_{1})$ & $1$\\
\bottomrule
\end{tabular}
\end{center}

\begin{figure}
\includegraphics[scale=0.75]{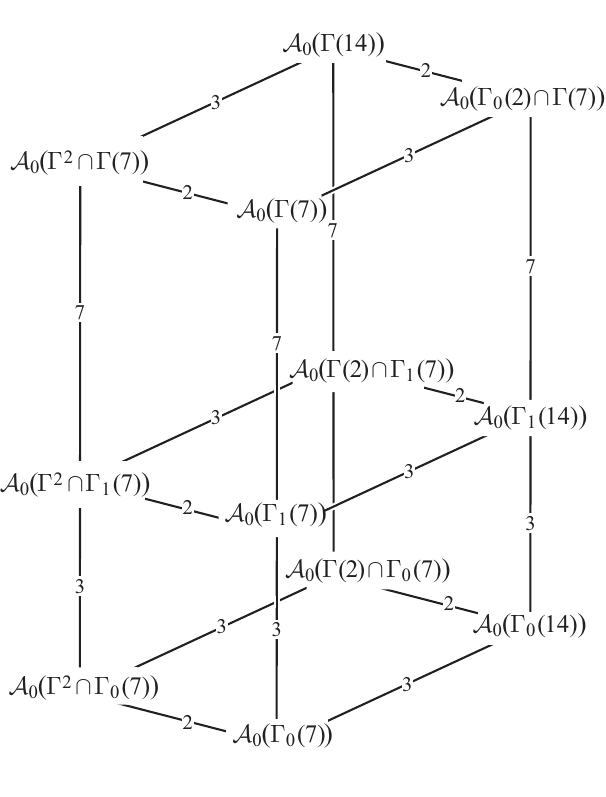}
\caption{Function fields between $\A_0(\Gamma_0(7))$ and $\A_0(\Gamma(14))$.}
\label{fig:function-fields}
\end{figure}

\subsection{Selected explicit models}

We now give explicit two-generator models for the five intermediate
curves listed above.  We begin with the two curves of largest genus.
In these two cases the plane models are rather singular, but their
normalizations and genera can be determined directly from the natural
coverings over the $\u$-line.

\begin{proposition}\label{prop:genus19-model}
Let $g=g_1$.  Then
\[
 \A_0(\G_0(2)\cap\G(7))
 =
 \C(\psi,\u,g)
 =
 \C(\psi,g).
\]
The functions $\psi$ and $g$ satisfy
\begin{equation}\label{eq:genus19-psig}
 (1-g)^5\psi^{14}
 +g(g-1)^3(g+1)Q(g)\psi^7
 +g^2(g+1)^5=0,
\end{equation}
where
\[
 Q(g)= 
 g^{10}+7g^9+23g^8+41g^7+37g^6+g^5
  -32g^4-30g^3-7g^2+3g+2.
\]
The normalization of the plane curve
\eqref{eq:genus19-psig} has genus $19$.
\end{proposition}

\begin{proof}
The function $g=g_1$ satisfies
\[
 g^3+g^2-\frac{\u^2-1}{\u}g
 +\frac{(\u-1)^2}{\u}=0,
\]
or equivalently
\begin{equation}\label{eq:g19-cubic}
 (1-g)\u^2
 +(g^3+g^2-2)\u
 +(g+1)=0.
\end{equation}
Together with $\psi^7=\u^4(\u-1)$,
elimination of $\u$ gives \eqref{eq:genus19-psig}.

Conversely, putting $P=\psi^7$, the first subresultant of these two
equations is linear in $\u$ and gives
\[
 \u=
 \frac{
 (g-1)^2P
 -(g+1)^2(g^3+g^2-1)(g^3+3g^2+4g+1)}
 {(g+1)R(g)},
\]
where
\[
 R(g)=
 g^9+5g^8+11g^7+9g^6-5g^5-17g^4
 -10g^3+2g^2+5g+1.
\]
Thus $\u\in\C(\psi,g)$, and hence
\[
 \C(\psi,\u,g)=\C(\psi,g).
\]

To compute the genus, first consider $K_1=\C(\u,g)$.
The degree-$3$ map corresponding to
$K_1/\C(\u)$ is branched at $0$, $1$, $\infty$,
and at the three roots of
\(
 \u^3-8\u^2+5\u+1
\).
Each of these six branch values has ramification type $(2,1)$.
Consequently, we have
\[
 2g(K_1)-2=3(-2)+6=0,
\]
and thus $g(K_1)=1$.

We now adjoin $\psi$, with $\psi^7=\u^4(\u-1)$.
Above each of $\u=0,1,\infty$, the degree-$3$ cover has
one ramified point of index $2$ and one unramified point.  Hence the
orders of $\u^4(\u-1)$ at the six points lying above
$0,1,\infty$ are respectively
\(
 8,\ 4; \ 2,\ 1; \ -10,\ -5
\).
None of these integers is divisible by $7$.  Thus the 
degree~$7$ cyclic extension $K_1(\psi)/K_1$ is totally ramified at precisely
these six points and unramified elsewhere.  Riemann--Hurwitz gives
\[
 2g-2 =
 7\bigl(2g(K_1)-2\bigr)+6(7-1) =36.
\]
Thus $g=19$.
\end{proof}

For the genus-$17$ curve it is convenient to use the following
normalization of the square root of the discriminant:
\begin{equation}\label{eq:Delta-normalization}
 \Delta^2
 =
 4\,\frac{(\u-1)^3
 (\u^3-8\u^2+5\u+1)}{\u^3}.
\end{equation}

\begin{proposition}\label{prop:genus17-model}
Put
\(
 x={\u\Delta}/({\u-1})\),
\(
 y={\psi}/{\u}
\).
Then, we have
\[
 \A_0(\Gtwo\cap\G(7))
 =\C(\psi,\u,\Delta)
 =\C(x,y).
\]
Moreover, $x$ and $y$ satisfy
\begin{multline}\label{eq:genus17-xy}
64y^{28}
-x^6y^{21}
-52x^4y^{21}
-1520x^2y^{21}
-18496y^{21} \\
-84x^4y^{14}
+5264x^2y^{14}
-3648y^{14} 
-1072x^2y^7
+64y^7
+16x^2=0.
\end{multline}
The normalization of the plane curve
\eqref{eq:genus17-xy} has genus $17$.
\end{proposition}

\begin{proof}
Set $t=y^7$.  Since $y=\psi/\u$, we have
$t={\psi^7}/{\u^7}
   ={\u-1}/{\u^3}$,
and hence
\begin{equation}\label{eq:g17-u1}
 t\u^3-\u+1=0.
\end{equation}
On the other hand, \eqref{eq:Delta-normalization} gives
$x^2
 =
 4\,(\u-1)
 (\u^3-8\u^2+5\u+1)/\u$,
so that
\begin{equation}\label{eq:g17-u2}
 x^2=
 4t\u^2
 (\u^3-8\u^2+5\u+1).
\end{equation}
Eliminating $\u$ from
\eqref{eq:g17-u1} and \eqref{eq:g17-u2}
gives \eqref{eq:genus17-xy}.

Conversely, the Euclidean algorithm applied to
\eqref{eq:g17-u1} and \eqref{eq:g17-u2},
regarded as polynomials in $\u$, gives
\[
 \u=
 -\frac{
 4(t^3x^2+16t^3-9t^2x^2+572t^2-256t+4)}
 {t^3x^4+32t^3x^2+464t^3
 +80t^2x^2-3280t^2+1040t-16}.
\]
Thus $\u\in\C(x,y)$.  Since
\(
 \psi=\u y\),
\( \Delta={x(\u-1)}/{\u},
\)
we obtain
\(
 \C(\psi,\u,\Delta)=\C(x,y)
\).

The cyclic degree-$7$ extension
\(
 \psi^7=\u^4(\u-1)
\)
of $\C(\u)$ is ramified precisely at
$\u=0,1,\infty$, with ramification index $7$.
The quadratic extension defined by
\eqref{eq:Delta-normalization} is ramified at
$0,1,\infty$ and at the three roots of
\[
 \u^3-8\u^2+5\u+1,
\]
since these are exactly the points where the rational function on the
right-hand side of \eqref{eq:Delta-normalization} has odd valuation.
Since $2$ and $7$ are coprime, the two extensions are linearly
disjoint over $\C(\u)$, and their compositum has degree $14$.
At $0,1,\infty$ the ramification index in the compositum is $14$,
whereas at the other three branch points it is $2$.  Hence
\[
 2g-2
 =
 14(-2)
 +3\cdot14\Bigl(1-\frac1{14}\Bigr)
 +3\cdot14\Bigl(1-\frac12\Bigr) 
 =32.
\]
Thus $g=17$.
\end{proof}

\begin{remark}\label{rem:g17-g19}
Propositions~\ref{prop:genus19-model} and
\ref{prop:genus17-model} give two-generator descriptions of the two
largest proper intermediate function fields considered above:
\[
 \A_0(\G_0(2)\cap\G(7))=\C(\psi,g_1),
 \qquad
 \A_0(\Gtwo\cap\G(7))=\C(x,y).
\]
In particular, the genera $19$ and $17$ are obtained without analyzing
the singularities of the corresponding plane models; they follow
directly from the ramification of the natural coverings over the
$\u$-line.
\end{remark}

We next record the lower-genus quotients, for which particularly simple
models are available.

\begin{proposition}\label{prop:C4}
The curve attached to $\G(2)\cap\G_1(7)$ has function field
\[
 \A_0(\G(2)\cap\G_1(7))
 =\C(\u,g_1,g_2)=\C(g_1,g_2),
\]
and the plane model
\begin{multline}\label{eq:C4}
 (Z_1+Z_2)(Z_1+Z_2+2)Z_1^2Z_2^2 \\
 +2(Z_1+Z_2+1)(Z_1+Z_2+2)Z_1Z_2
 -(Z_1+Z_2)^2(Z_1+Z_2+1)^2=0.
\end{multline}
It has genus $4$.
\end{proposition}

\begin{proof}
By the fixed-field description in Section~6.2,
\[
 A_0(\Gamma(2)\cap\Gamma_1(7))
   =\mathbf C(\upsilon,g_1,g_2).
\]
Put $Z_1=g_1$, $Z_2=g_2$, $Z_3=-1-Z_1-Z_2$.
By Theorem~\ref{thm:intro-theta}, if
$s_2=Z_1Z_2+Z_2Z_3+Z_3Z_1$,
$s_3=Z_1Z_2Z_3$,
then $s_2=-\upsilon+1/\upsilon$,
$s_3=-\upsilon-1/\upsilon+2$.
Hence we have
$\upsilon=({2-s_2-s_3})/{2}\in\mathbf C(Z_1,Z_2)$,
and therefore $\C(\upsilon,g_1,g_2)=\C(g_1,g_2)$.

Eliminating $\upsilon$ from the two displayed equations for $s_2$ and
$s_3$, or equivalently using $(2-s_2-s_3)(2+s_2-s_3)=4$,
and then substituting $Z_3=-1-Z_1-Z_2$, gives
\begin{multline*}
 (Z_1+Z_2)(Z_1+Z_2+2)Z_1^2Z_2^2
\\
 +2(Z_1+Z_2+1)(Z_1+Z_2+2)Z_1Z_2
 -(Z_1+Z_2)^2(Z_1+Z_2+1)^2=0,
\end{multline*}
which is \eqref{eq:C4}.

The group is torsion-free and has index
$[\mathrm{PSL}_2(\mathbf Z):\overline{\Gamma(2)\cap\Gamma_1(7)}]
=144$.
It has $18$ cusps.  Indeed, modulo $14$ the quotient over
$\Gamma(14)$ is the order-$7$ unipotent subgroup of
$\mathrm{SL}_2(\mathbf F_7)$; its action on primitive vectors modulo
$14$, up to sign, has $3(3+3)=18$ orbits.
Hence $g=1+{144}/{12}-{18}/{2}=4$.
\end{proof}

\begin{proposition}\label{prop:C2}
The curve attached to $\Gtwo\cap\G_1(7)$ has model
\begin{equation}\label{eq:C2}
 C_2: \ 
 W^2=X(X-1)(X^3-8X^2+5X+1),
\end{equation}
and genus $2$.
\end{proposition}

\begin{proof}
By Section~6.2, the fixed field corresponding to
$\Gamma^2\cap\Gamma_1(7)$ is
\[
 A_0(\Gamma^2\cap\Gamma_1(7))
   =\mathbf C(\upsilon,\delta).
\]
From \eqref{eq:delta-square} we have
\[
 \delta^2
 =\upsilon(\upsilon-1)
   \bigl(\upsilon^3-8\upsilon^2+5\upsilon+1\bigr).
\]
Thus, on putting
$X=\upsilon$, $W=\delta$,
we obtain the model
\[
 W^2=X(X-1)(X^3-8X^2+5X+1).
\]
The polynomial on the right is square-free of degree $5$, so its smooth
projective model is hyperelliptic of genus
$(5-1)/2=2$.\end{proof}

\begin{proposition}\label{prop:X114}
The modular curve $X_1(14)$ is generated by
$\u$ and $g_1$.
After
\[
 X_0=-\frac1{g_1},
 \quad
 Y_0=
 \frac{(g_1-1)(-g_1^2+\u-g_1-1)}{g_1^3},
\]
one obtains the minimal Weierstrass equation
\begin{equation}\label{eq:X114}
 Y_0^2+X_0Y_0+Y_0=X_0^3-X_0.
\end{equation}
\end{proposition}

\begin{proof}
Again by Section~6.2,
\[
 A_0(\Gamma_1(14))=\mathbf C(\upsilon,g_1).
\]
Writing $g=g_1$, equation \eqref{eq:g19-cubic} is
\[
 (1-g)\upsilon^2+(g^3+g^2-2)\upsilon+(g+1)=0.
\]
Set $X_0=-1/g$, $Y_0={(g-1)(-g^2+\upsilon-g-1)}/{g^3}$.
A direct substitution transforms the preceding equation into
\[
 Y_0^2+X_0Y_0+Y_0=X_0^3-X_0.
\]
Conversely, we have 
$g=-1/{X_0}$,
 $\upsilon={X_0^3+Y_0+1}/({X_0^2(X_0+1)})$,
so the transformation is birational.  Hence this equation has function
field $\mathbf C(\upsilon,g_1)$ and therefore gives a Weierstrass model
of $X_1(14)$.

Its discriminant is $-28$; in particular this integral Weierstrass
equation is minimal (equivalently, this follows immediately from the
standard minimality criterion, or from Tate's algorithm).  Thus
\[
 Y_0^2+X_0Y_0+Y_0=X_0^3-X_0
\]
is a minimal Weierstrass equation for $X_1(14)$.
\end{proof}

\begin{remark}
The involution $(Z_1,Z_2)\mapsto(Z_2,Z_1)$ on \eqref{eq:C4}
gives the elliptic quotient $X_1(14)$.
The cyclic permutation
\[
 (Z_1,Z_2)\longmapsto (Z_2,-1-Z_1-Z_2)
\]
gives the degree-$3$ quotient \eqref{eq:C2}.
These maps make the subgroup lattice of $S_3$ visible geometrically.
Thus the five models above exhibit two complementary features of the
intermediate-field structure: the genus-$19$ and genus-$17$ curves are
naturally understood through ramified coverings of the $\u$-line,
whereas the genus-$4$, genus-$2$, and genus-$1$ quotients display the
$S_3$-symmetry particularly explicitly.
\end{remark}

\subsection{Remark on $\G_{0}(7)$}

The generator of $\A_0(\Gamma(1))$ is the famous elliptic modular $j$-function
\[ 
j(\tau)=q^{-1}+744+196884\thin q+21493760\thin q^2+\cdots. 
\] 
The generators of $\A_{0}(\G_{0}(7))$ and $\A_{0}(\G_{0}(14))$ are more or less well known. Let 
\begin{align*}
j_7(\tau)&:={\eta(\tau)^4}/{\eta(7\tau)^4}
=q^{-1} - 4 + 2q + 8 q^2 - 5 q^3 - 4 q^4 - 10 q^5 + 12 q^6 - \cdots,\\
j_{14}(\tau)&:=\frac{\eta(2\tau)\eta(7\tau)^7}{\eta(\tau)\eta(14\tau)^7}
=q^{-2}+q^{-1}+1+2q+2\thin q^2+3\thin q^3 + 4 q^4 - 2 q^5 - q^6 + \cdots,
\end{align*}
where $\eta(\tau)=q^{1/24}\prod_{n=1}^\infty (1-q^n)$ is the Dedekind eta function.  Then we have $\A_{0}(\G_{0}(7))=\C(j_{7})$, and $\A_{0}(\G_{0}(14))=\C(j_{7},j_{14})$. See for instance~\cite{Lee-Park} and~\cite{Duke}.
The relation between $j_7(\tau)$ and $ j_{14}(\tau)$ is also classical:
\begin{equation}\label{eq:j7-j14}
j_{14}^3 - (j_7^2 + 9 j_7 + 17) j_{14}^2 + 16 (j_7 + 5) j_{14} - 64 = 0.
\end{equation}
It is easy to see that this equation defines an elliptic curve. 
It is known that the degree-$8$ cover $X_{0}(7)/X(1)$ is given by 
\begin{equation}\label{eq:j1-j7}
j = \frac{(j_{7}^{2}+13j_{7}+49)(j_{7}^{2}+245j_{7}+2401)^{3}}{j_{7}^{7}},
\end{equation}
and the degree-$3$ cover $X_{1}(7)/X_{0}(7)$ is given by
\begin{equation}\label{eq:j7-u}
j_{7}= \u + \frac{1}{1-\u} + \frac{\u-1}{\u} -8 
=\frac{\u^{3} - 8\u^{2} + 5\u +1}{\u(\u-1)}.
\end{equation}
Also, the degree-$3$ cover $X_{1}(14)/X_{0}(14)$ is given by
\begin{equation}\label{eq:j14-g1}
j_{14}= \frac{8g_1}{(g_1-1)(g_1+1)^2}.
\end{equation}

% ----------------------------------------------------------------
\section{A complementary level-$49$ picture}
\label{sec:level49}
% ----------------------------------------------------------------

This section is not used in the proof of the level-$14$ results.
We include it to record a complementary quotient construction at level
$49$, in which the same septic theta functions give explicit
coordinates on the CM elliptic curve $X_0(49)$.

By Lemma~\ref{lem:upsilon-hauptmodul}, $\u$ is a Hauptmodul for
$\G_1(7)$.  The corresponding universal elliptic curve, with its
distinguished point of order $7$, is
\begin{equation}\label{eq:X17-univ}
 y^2-(\u^2-\u-1)xy-\u^2(\u-1)y
 =
 x^3-\u^2(\u-1)x^2,
\end{equation}
with distinguished point $(0,0)$.

The quotient construction associated with a full level-$7$ structure
gives a moduli-theoretic identification
\begin{equation}\label{eq:X7-X149}
 X(7)\simeq
 X\bigl(\G_1(7)\cap\G_0(49)\bigr).
\end{equation}
At the level of modular functions the two groups are conjugate in
$\mathrm{SL}_2(\mathbb R)$, and the identification corresponds to
$\tau\mapsto7\tau$.  Thus
\[
 \A_0(\G(7))=\C(\psi(\tau),\u(\tau)),
\]
whereas
\[
 \A_0(\G_1(7)\cap\G_0(49))
 =
 \C(\psi(7\tau),\u(7\tau)).
\]
We now describe explicitly the quotient from
$X(\G_1(7)\cap\G_0(49))$ to $X_0(49)$.

\begin{proposition}\label{prop:X049}
Put
%\begin{equation}\label{eq:p1p2}
$p_1(\tau)=-\phi_1(7\tau)$,
% \quad
 $p_2(\tau)=\phi_2(7\tau)$,
%\end{equation}
and define
\begin{equation}\label{eq:X049-XY}
 X=
 1+p_1+p_2-\frac1{p_1p_2},
 \quad
 Y=
 \frac{(p_1+1)(p_2+1)(p_1p_2-1)}{p_1p_2}.
\end{equation}
Then
\(
 X,Y\in\A_0(\G_0(49))
\),
and they satisfy
\begin{equation}\label{eq:X049-minimal}
 Y^2+XY=X^3-X^2-2X-1.
\end{equation}
Furthermore,
\[
 \A_0(\G_0(49))=\C(X,Y).
\]
Thus \eqref{eq:X049-minimal} is a minimal Weierstrass model of
$X_0(49)$.
\end{proposition}

\begin{proof}
Let
\(
 H=\G_1(7)\cap\G_0(49), G=\G_0(49)
\).
Since $-I\in\SL_{2}(\Z)$ acts trivially on the upper half-plane, the effective
quotient is
\(
 G/\langle -I,H\rangle\simeq C_3.
\)
For example, we choose the matrix
\(
 \gamma=
 \sltwo(25,-1;-49,2)\in G
\).
After putting $z=7\tau$, its action becomes
\[
 z\longmapsto\frac{25z-7}{-7z+2},
\]
and hence is represented modulo $7$ by
\(
 \sltwo(4,0;0,2)
\),
an element of order $3$ in $\mathrm{PSL}_2(\F_7)$.

For notational clarity, put
\[
 A_i(\tau):=a_i(7\tau)\qquad (i=1,2,3).
\]
Thus $(A_1:A_2:A_3)$ lies on the Klein quartic
$A_1^3A_2-A_2^3A_3-A_3^3A_1=0$,
and
\[
 p_1=A_3/A_2,
 \quad
 p_2=-A_2/A_1.
\]
Applying the standard transformation formula for theta functions with
rational characteristics, recalled in
Appendix~\ref{sec:appendix-theta}, to the above matrix gives, up to a
common non-zero factor,
\begin{equation}\label{eq:C3-action-ai}
 (A_1,A_2,A_3)\longmapsto(A_3,-A_1,A_2).
\end{equation}
It follows that a generator $\sigma$ of the effective quotient acts by
\begin{equation}\label{eq:C3-action-p}
 \sigma(p_1,p_2)
 =
 \Bigl(p_2,-\frac1{p_1p_2}\Bigr).
\end{equation}
Thus
\(
 p_1, p_2, -\frac{1}{p_1p_2}
\)
are cyclically permuted.
So, let
\(
 r_1=p_1,
 r_2=p_2,
 r_3=-\frac1{p_1p_2}
\).
Then
\[
 X=1+r_1+r_2+r_3,
\]
and, since $r_1r_2r_3=-1$, we have
\[
 Y=(1+r_1)(1+r_2)(1+r_3).
\]
Hence $X$ and $Y$ are invariant under the quotient action and therefore
belong to $\A_0(G)$.

The Klein relation for $p_1,p_2$ is
\begin{equation}\label{eq:Klein-p1p2}
 p_1^3p_2^2-p_1p_2^3=1.
\end{equation}
Substituting \eqref{eq:X049-XY} into the left-hand side of
\eqref{eq:X049-minimal}, and using \eqref{eq:Klein-p1p2}, gives
\[
 Y^2+XY-X^3+X^2+2X+1=0.
\]
Thus $X$ and $Y$ satisfy \eqref{eq:X049-minimal}.

For completeness, let
\[
 s_1=r_1+r_2+r_3=X-1,\quad
 s_2=r_1r_2+r_2r_3+r_3r_1=Y-X+1.
\]
Then $r_1,r_2,r_3$ are the roots of
\begin{equation}\label{eq:cubic-r123}
 T^3-(X-1)T^2+(Y-X+1)T+1=0.
\end{equation}
For a generic point of the curve \eqref{eq:X049-minimal}, the three
points lying above it are precisely the cyclic orbit
\[
 (p_1,p_2),\quad
 \Bigl(p_2,-\frac1{p_1p_2}\Bigr),\quad
 \Bigl(-\frac1{p_1p_2},p_1\Bigr).
\]
Consequently the map has degree $3$, equal to
$[G:\langle -I,H\rangle]$, and hence
\[
 \A_0(G)=\C(X,Y). \qedhere
\]
\end{proof}

\begin{remark}\label{rem:X049-cubic}
Writing $r=p_1p_2$, one also obtains the cubic
\[
 r^3+(X-Y-1)r^2+(1-X)r-1=0.
\]
Its three roots are
\(
 p_1p_2,  -\frac1{p_1}, -\frac1{p_2}
\).
This gives another explicit manifestation of the cyclic degree-$3$
quotient from \eqref{eq:X7-X149} to $X_0(49)$.
\end{remark}

\begin{remark}
The elliptic curve \eqref{eq:X049-minimal} has complex multiplication by
\(
 \Z\Bigl[\frac{1+\sqrt{-7}}2\Bigr]
\).
Thus the same septic theta functions which describe the full
level-$7$ curve also give, after replacing $\tau$ by $7\tau$ and taking
the invariant combinations \eqref{eq:X049-XY}, explicit coordinates on
the CM elliptic curve $X_0(49)$.  In this sense
Proposition~\ref{prop:X049} gives a concrete theta-functional form of
the quotient
\[
 X\bigl(\G_1(7)\cap\G_0(49)\bigr)\longrightarrow X_0(49).
\]

The elliptic surface \eqref{eq:X17-univ} is a singular $K3$ surface
and is related to the corresponding Inose surface.  We do not pursue
this direction here.
\end{remark}

% ----------------------------------------------------------------
\appendix
\section{Theta identities for the degree-$7$ model}
\label{sec:appendix-theta}
% ----------------------------------------------------------------

We give enough details here to explain the origin of the fourteen
quadrics \eqref{eq:KV-14quadrics}.  The calculation is classical and
starts from Jacobi's four-term ``main identity'' \cite{Jacobi:Theta}.  In the notation
\eqref{eq:theta-k}, it may be written as
\begin{multline}\label{eq:Jacobi-main}
\theta_{7/2}(w)\theta_{7/2}(x)\theta_{7/2}(y)\theta_{7/2}(z)
 -\theta_0(w)\theta_0(x)\theta_0(y)\theta_0(z) 
 \\
 =\theta_{7/2}(w')\theta_{7/2}(x')\theta_{7/2}(y')\theta_{7/2}(z')
 -\theta_0(w')\theta_0(x')\theta_0(y')\theta_0(z'),
\end{multline}
where
\[
\begin{aligned}
 w'&=\tfrac12(w+x+y+z),&
 x'&=\tfrac12(w+x-y-z),\\
 y'&=\tfrac12(w-x+y-z),&
 z'&=\tfrac12(w-x-y+z).
\end{aligned}
\]
The following translation and parity formulas follow directly
from the definition \eqref{eq:theta-k}:
\begin{equation}\label{eq:theta-trans-app}
\setlength{\arraycolsep}{2pt}
\renewcommand{\arraystretch}{1.2}
\begin{array}{lcrl}
 \theta_k(z+1)&=&(-1)^{2k+1}&\theta_k(z),\\
 \theta_k(z+\tau)&=&-e^{-7\pi i\tau-14\pi iz}\,&\theta_k(z),\\
 \theta_k(z+\frac{\tau}{14})
   &=&-ie^{-\pi i\tau/28-\pi iz}\,&\theta_{k-\frac{1}{2}}(z),\\
	 \theta_k(z+\frac{\tau}{7})
   &=&-e^{-\pi i\tau/7-2\pi iz}\,&\theta_{k-1}(z),\\
 \theta_k(z+\frac{2\tau}{7})
   &=&e^{-4\pi i\tau/7-4\pi iz}\,&\theta_{k-2}(z),\\
 \theta_k(-z)&=&(-1)^{2k+1}&\theta_{-k}(z).
\end{array}
\end{equation}
In particular,
\[
 \theta_0(0)=0,
 \quad \theta_4(0)=-\theta_3(0),
 \quad \theta_5(0)=-\theta_2(0),
 \quad \theta_6(0)=-\theta_1(0).
\]
All suffixes below are understood modulo $7$.

\subsection{A basic addition formula}

Applying the translations in \eqref{eq:theta-trans-app} to
\eqref{eq:Jacobi-main}, and then taking the same linear combination as
in Jacobi's derivation of the addition theorem, gives
\begin{multline}\label{eq:basic-addition}
 \theta_3(0)\theta_4(x+y)\theta_4(y+z)\theta_4(z+x)
 \\
=\theta_1(x+y+z)\theta_0(x)\theta_0(y)\theta_0(z)
 -\theta_4(x+y+z)\theta_4(x)\theta_4(y)\theta_4(z).
\end{multline}
For reference, one way to obtain \eqref{eq:basic-addition} is as
follows.  Translate $w$ by $\tau/7$ in
\eqref{eq:Jacobi-main}; translate all four variables six times by
$\tau/14$; and in the resulting three identities put respectively
$w=-(x+y+z)$, $w=x+y+z$, and $w=-(x+y+z)$.  Taking one half of
$-$(the first)$+$(the second)$+$(the third) yields
\eqref{eq:basic-addition}.  This is the only four-variable identity we
shall need below.

Putting $z=-y$ in \eqref{eq:basic-addition} gives
\begin{equation}\label{eq:addition-I0}
 \theta_3(0)^2\theta_4(x-y)\theta_4(x+y)
 =\theta_0(x)\theta_1(x)\theta_0(y)^2
  -\theta_4(x)^2\theta_3(y)\theta_4(y).
\end{equation}
Translating both $x$ and $y$ successively by $\tau/7$ gives the cyclic
family
\begin{multline}\label{eq:addition-I}
\theta_3(0)^2\theta_4(x-y)\theta_{4-2i}(x+y)
 \\
 =\theta_{6-i}(x)\theta_{1-i}(x)\theta_{6-i}(y)^2
 -\theta_{4-i}(x)^2\theta_{3-i}(y)\theta_{4-i}(y),
\quad i\in\Z/7\Z.
\end{multline}
Translating only $x$ gives the six companion families obtained by
replacing the first suffix $4$ on the left successively by
$3,2,1,0,6,5$.  These are the first set of septic addition formulas.

A second family is obtained by translating $x,y,z$ simultaneously by
$\tau/7$ in \eqref{eq:basic-addition}.  After again putting $z=-y$ one
finds
\begin{equation}\label{eq:addition-II0}
\theta_2(0)\theta_3(0)\theta_2(x-y)\theta_2(x+y)
=\theta_1(x)\theta_3(x)\theta_3(y)\theta_4(y)
 -\theta_5(x)\theta_6(x)\theta_1(y)\theta_6(y).
\end{equation}
Hence, by translating $x$ and $y$ together,
\begin{multline}\label{eq:addition-II}
\theta_2(0)\theta_3(0)
 \theta_{2-i}(x-y)\theta_{2-2i}(x+y)
 \\
 =\theta_{1-i}(x)\theta_{3-i}(x)
 \theta_{3-i}(y)\theta_{4-i}(y)
 -\theta_{5-i}(x)\theta_{6-i}(x)
 \theta_{1-i}(y)\theta_{6-i}(y),
\quad i\in\Z/7\Z.
\end{multline}
Again, translating only $x$ gives the companion formulas.  Equations
\eqref{eq:addition-I} and \eqref{eq:addition-II}, together with these
cyclic translates, are the addition formulas needed for the projective
model.

\subsection{Derivation of the fourteen quadrics}

We now set $y=0$ in the preceding addition formulas.  Using the parity
relations at $0$ and deleting repetitions, the first family becomes
\begin{equation}\label{eq:quad-family-I}
 \theta_1(0)\theta_2(0)\theta_i(x)^2
 =\theta_2(0)^2\theta_{i+3}(x)\theta_{i-3}(x)
  -\theta_3(0)^2\theta_{i+2}(x)\theta_{i-2}(x),
 \qquad i\in\Z/7\Z,
\end{equation}
and the second becomes
\begin{equation}\label{eq:quad-family-II}
 \theta_2(0)\theta_3(0)\theta_i(x)^2
 =\theta_1(0)^2\theta_{i+3}(x)\theta_{i-3}(x)
  -\theta_3(0)^2\theta_{i+1}(x)\theta_{i-1}(x),
 \quad i\in\Z/7\Z.
\end{equation}
For example, the case $i=0$ reads
\begin{align*}
 \theta_1(0)\theta_2(0)\theta_0(x)^2
 &=\theta_2(0)^2\theta_3(x)\theta_4(x)
   -\theta_3(0)^2\theta_2(x)\theta_5(x),\\
 \theta_2(0)\theta_3(0)\theta_0(x)^2
 &=\theta_1(0)^2\theta_3(x)\theta_4(x)
   -\theta_3(0)^2\theta_1(x)\theta_6(x).
\end{align*}
Thus, on writing
\[
 a_j=\theta_j(0)\quad(j=1,2,3),
 \qquad x_i=\theta_i(x),
\]
Equations \eqref{eq:quad-family-I} and \eqref{eq:quad-family-II} are
precisely the fourteen quadrics \eqref{eq:KV-14quadrics}.
This proves the assertion used in Proposition~\ref{prop:theta-embedding}
and, in particular, explains the origin of equation
\eqref{eq:KV-14quadrics} rather than merely quoting it.

Finally, setting $x=0$ in the fourteen equations recovers
Klein's quartic relation 
\[
 \theta_1(0)^3\theta_2(0)
 =\theta_2(0)^3\theta_3(0)
  +\theta_3(0)^3\theta_1(0),
\]
which is exactly \eqref{eq:Klein-a}.

% ----------------------------------------------------------------
\section*{Acknowledgements}
% ----------------------------------------------------------------

In preparing this manuscript, the authors made use of ChatGPT (OpenAI) for assistance with language editing, 
exposition, and some computational checks. All mathematical arguments and computations were independently verified by the authors.

M.~Kaneko was supported by JSPS KAKENHI Grant Numbers JP21K18141
and JP21H04430.  M.~Kuwata was supported by the Chuo University Grant
for Special Research.

%\bibliography{KKlevel14.bib}
%\bibliographystyle{amsplain}

\providecommand{\bysame}{\leavevmode\hbox to3em{\hrulefill}\thinspace}
\providecommand{\MR}{\relax\ifhmode\unskip\space\fi MR }
\providecommand{\MRhref}[2]{%
  \href{http://www.ams.org/mathscinet-getitem?mr=#1}{#2}
}
\providecommand{\href}[2]{#2}

\end{document}